\documentclass[a4paper, 12pt]{amsart}

\usepackage{ucs}
\usepackage[utf8x]{inputenc}
\usepackage[english]{babel}
\usepackage{amsfonts,amsmath, amssymb}
\usepackage{fancyhdr}
\usepackage[fixlanguage]{babelbib}
\usepackage[arrow, matrix, curve]{xy}
\usepackage{hyperref}
\usepackage{color}
\usepackage{mathtools, tensor, faktor}
\usepackage{cite}
\usepackage{esint}
\usepackage{cancel}

\newtheorem{thm}{Theorem}[section]
\newtheorem{lem}[thm]{Lemma}
\newtheorem{cor}[thm]{Corollary}
\newtheorem{prop}[thm]{Proposition}

\theoremstyle{definition}
\newtheorem{defi}[thm]{Definition}
\newtheorem{ex}[thm]{Example}

\theoremstyle{remark}
\newtheorem{rem}[thm]{Remark}

\newcommand{\eps}{\varepsilon}
\newcommand{\pin}{$\text{pin}^-$}
\newcommand{\mfdspace}{\mathcal{M}(n,d)}
\newcommand{\mfdspaceCn}{\mathcal{M}(n,d,C)}
\newcommand{\mfdspaceCnn}{\mathcal{M}(n+1,d,C)}
\newcommand{\opspace}{\mathcal{O}(n,d,C)}

\newcommand{\dGH}{d_{\text{GH}}}
\newcommand{\bbN}{\mathbb{N}}
\newcommand{\bbS}{\mathbb{S}}
\newcommand{\bbR}{\mathbb{R}}
\newcommand{\bbC}{\mathbb{C}}
\newcommand{\bbZ}{\mathbb{Z}}

\newcommand{\Ma}{M_a}
\newcommand{\ga}{g_a}
\newcommand{\tga}{\tilde{g}_a}
\newcommand{\ha}{h_a}
\newcommand{\fa}{f_a}
\newcommand{\Aa}{A_a}
\newcommand{\Ta}{T_a}
\newcommand{\Fa}{F_a}
\newcommand{\cFa}{\mathcal{F}_a}
\newcommand{\la}{l_a}
\newcommand{\Ra}{R_a}
\newcommand{\Za}{Z_a}
\newcommand{\Ka}{K_a}
\newcommand{\tZa}{\tilde{Z}_a}
\newcommand{\cZa}{\mathcal{Z}_a}
\newcommand{\phia}{\varphi_a}
\newcommand{\Da}{D_a}
\newcommand{\tDa}{\tilde{D}_a}

\DeclareMathOperator{\inj}{inj}
\DeclareMathOperator{\diam}{diam}
\DeclareMathOperator{\vol}{vol}
\DeclareMathOperator{\grad}{grad}
\DeclareMathOperator{\dvol}{dvol}
\DeclareMathOperator{\arsinh}{arsinh}
\DeclareMathOperator{\Hom}{Hom}
\DeclareMathOperator{\dist}{dist}
\DeclareMathOperator{\GL}{GL}
\DeclareMathOperator{\Spin}{Spin}
\DeclareMathOperator{\SO}{SO}
\DeclareMathOperator{\D}{D}
\DeclareMathOperator{\Aff}{Aff}
\DeclareMathOperator{\Pin}{Pin}

\begin{document}

\title[Dirac operators under codimension one collapse]{Dirac operators with $W^{1,\infty}$-potential under codimension one collapse}
\author{Saskia Roos}
\address{Max-Planck-Institut für Mathematik, Vivatsgasse 7, 53111 Bonn, Germany}
\email{saroos@mpim-bonn.mpg.de}
\begin{abstract}
We study the behavior of the spectrum of the Dirac operator together with a symmetric $W^{1, \infty}$-potential on spin manifolds under a collapse of codimension one with bounded sectional curvature and diameter. If there is an induced spin structure on the limit space $N$ then there are convergent eigenvalues which converge to the spectrum of a first order differential operator $D$ on $N$ together with a symmetric $W^{1,\infty}$-potential. In the case of an orientable limit space $N$, $D$ is the spin Dirac operator $D^N$ on $N$ if the dimension of the limit space is even and if the dimension of the limit space is odd, then $D = D^N \oplus -D^N$.
\end{abstract}
\maketitle

\section{Introduction}

After studying the structure of collapsing sequences of manifolds under bounded sectional curvature and diameter, done by Cheeger, Fukaya and Gromov, ( see \cite{CheegerFukayaGromov} and the references therein) one of the next questions arising was how the spectrum of differential operators behaves in the limit of a collapsing sequence.

As for the Laplacian on functions, Fukaya showed in \cite{FukayaLaplace} that if a sequence of manifolds with uniform bounded sectional curvature and diameter converges in  the measured Gromov-Hausdorff-topology, then the eigenvalues of the Laplace operator converge to the eigenvalues of the Laplacian on the limit space with respect to a limit measure even in the case that the limit happens to be a smooth manifold. This result was generalized to the Laplacian on $p$-forms by Lott in \cite{LottLaplaceSmooth}, \cite{LottLaplaceSingular}. Using the Bochner-type formula for Dirac operators on $G$-Clifford bundles on manifolds, where $G \in \lbrace \SO(n), \Spin(n) \rbrace$, Lott proved similar results for Dirac eigenvalues under collapse with bounded sectional curvature and diameter \cite{LottDirac}. His results also include the Dirac operator acting on differential forms considering the measured Gromov-Hausdorff topology.

In this paper, we consider sequences $(\Ma, \ga)_{a \in \bbN}$ of spin manifolds with bounded sectional curvature and diameter such that their Gromov-Hausdorff limit $(N,h)$ has codimension one. This already implies that $N$ is a Riemannian orbifold, see \cite[Proposition 11.5]{FukayaOrbifold}. By restricting to the setting of spin manifolds we are able to show that the spectrum of the Dirac operator together with a uniform bounded symmetric $W^{1, \infty}$-potential converges again to the Dirac operator with symmetric $W^{1, \infty}$-potential on the limit space $N$, where the Dirac operator is taken with respect to the standard measure $\dvol(h)$. In particular, we do not need to consider the limit measure in the measured Gromov-Hausdorff topology, as in \cite{LottDirac}, which is in general different to the standard measure $\dvol(h)$. We restrict ourselves to the case of collapse of codimension one as for higher codimensions the situation is more complicated, see Remark \ref{HigherCodimension}. 

We use the techniques of \cite{Ammann}, where Dirac operators on collapsing $\bbS^1$-principal bundles were considered. One of the main differences to \cite{Ammann} is that he assumes the norm of the curvature of the $\bbS^1$-principal bundle times the length of the fiber to vanish in the limit. This assumption is  not fulfilled for general collapsing $\bbS^1$-principal bundles with bounded sectional curvature and diameter, see Example \ref{NonVanishingA}. However, removing this assumption leads to an additional zero-order term in the limit. In addition, the limit space of a collapsing sequence of spin manifolds can happen to be nonorientable. In that case, we have to deal with an $\bbS^1$-bundle with affine structure group which is not necessarily an $\bbS^1$-principal bundle. 

We consider Dirac operators  with symmetric $W^{1, \infty}$-potential $\Za$ and show that in the limit the spectrum  of $\Da + \Za$ converges to the spectrum of a Dirac operator with a $W^{1, \infty}$-potential on the limit space. Furthermore, we show that, similar to \cite{Ammann} and \cite{LottDirac}, there are only convergent eigenvalues if and only if the spin structure on the manifolds $(\Ma,\ga)$ induce the same spin rep.\ \pin structure on the limit space for all $a \in \bbN$.

The paper is structured as follows. First we recall the structure of collapse of codimension one from \cite{Roos}. Then we discuss the notion of a projectable spin structure, which was first formulated by \cite{MoroianuDiss} for general $G$-principal bundles over manifolds. As, in that case, any sufficiently collapsed manifold is the total space of an $\bbS^1$-orbifold bundle over the limit space $N$, we extend the notion of projectable spin structures appropriately. In the next section, the behavior of $W^{1, \infty}$-bounded operators on spin manifolds under collapse of codimension one with bounded curvature and diameter is discussed. In particular we show under which circumstances one obtains convergence to an operator on the limit space. Combining everything, we prove the convergence results for Dirac operators with symmetric $W^{1, \infty}$-potential under collapse of codimension one. Here we consider the cases of nonprojectable and projectable spin structures separately. In the last section we discuss the special case of Dirac operators without a potential and relate them to the results of \cite{Ammann} and \cite{LottDirac}.

\subsection*{Acknowledgments}
I would like to thank Bernd Ammann and Werner Ballmann for their continuous support and enlightening discussions. Furthermore, I am very grateful for the hospitality and the support of the Max-Planck Institute for Mathematics in Bonn.

\section{Codimension one collapse}

Let $\mfdspace$ be the space of all closed $n$-dimensional Riemannian manifolds with $\diam(M) \leq d$ and $\vert \sec \vert \leq 1$. In \cite{Roos} we introduced the subspace
\begin{align*}
\mfdspaceCn \coloneqq \left\lbrace (M,g) \in \mfdspace : C \leq \frac{\vol(M)}{\inj(M)} \right\rbrace
\end{align*}
and showed the following properties.

\begin{thm}\label{CollapseSpace}
Let $(\Ma, \ga)_{a \in \bbN}$ be a sequence in $\mfdspaceCn$ which Gromov-Hausdorff converges to a lower dimensional compact metric space $N$. Then
\begin{enumerate}
	\item $N$ is an $(n-1)$-dimensional Riemannian orbifold with a $C^{1, \alpha}$-metric $h$.
	\item $\vol(N) \geq V$ for some positive constant $V \coloneqq V(n,d,C)$.
	\item $\Vert \sec(N) \Vert_{L^{\infty}} \leq K$ for some positive constant $K \coloneqq K(n,d,C)$.
\end{enumerate} 
\end{thm} 

We always talk about the limit space $N$ being a Riemannian orbifold where Riemannian manifolds are included as a special case. For background material about Riemannian orbifold and orbifold bundles, we refer to \cite[Chapter 4]{Boyer} and \cite[Chapter 13]{Thurston}.

By Fukaya's fibration theorem \cite{FukayaCollapse1}, applied to the space $\mfdspaceCn$ there is a constant $\eps(n,d)$ such that if the Gromov-Hausdorff distance between $(M,g) \in \mfdspaceCn$ and  some lower dimensional space $N$ in the $\dGH$-closure of $\mfdspaceCn$ is smaller than $\eps(n,d,C)$ then there is a map $f: M \rightarrow N$ such that $(M,N,f)$ is an $\bbS^1$-orbifold bundle with affine structure group. At this point we want to recall the following theorem proven by Cheeger, Fukaya and Gromov in \cite{CheegerFukayaGromov}, which we adjusted here to our setting.

\begin{thm}[{\cite{CheegerFukayaGromov}}]\label{invariantMetric}
Let $(M,g) \in \mfdspaceCn$ and assume that there is a $(n-1)$-dimensional Riemannian orbifold $N$ with $\dGH(M,N) \leq \eps(n,d)$. Then $(M,g)$ is an $\bbS^1$-orbifold bundle with structure group in $\Aff(\bbS^1)$. Furthermore there is a so-called invariant metric $\tilde{g}$ such that $\Vert \tilde{g} - g\Vert_{C^1} \leq C(n) \dGH(M,N)$ and such that $\bbS^1$ acts locally by isometries on $(M,\tilde{g})$. In particular, if $M$ and $N$ are orientable then $M \rightarrow N$ is an $\bbS^1$-principal orbifold bundle.
\end{thm}

\begin{rem}
As we are interested of collapsing sequences of spin manifolds, we only deal with the case of collapsing sequences of orientable manifolds.
\end{rem}

If the limit space $N$ is nonorientable then we can consider its orientation covering $\hat{N}$ and the pullback bundle $\hat{M}$ of the $\bbS^1$-orbifold bundle $M$. Since the structure group bundle $M \rightarrow N$ lies in $\Aff(\bbS^1) \cong\bbS^1 \rtimes \lbrace -1 , 1 \rbrace$ it follows that $\hat{M} \rightarrow \hat{N}$ is an $\bbS^1$-principal orbifold bundle. 

For simplicity we consider the case of an orientable limit space $N$ and explain, if needed, the modifications for the nonorientable case. In that case any sufficiently collapsed manifold in $\mfdspaceCn$ is an $\bbS^1$-principal orbifold bundle over its limit space. Moreover, by Theorem \ref{invariantMetric} $\bbS^1$ acts on $(M, \tilde{g})$ isometrically for a nearby metric $\tilde{g}$.

For such an $\bbS^1$-principal orbifold bundle $f: (M,g) \rightarrow (N,h)$ with $f$ being a Riemannian submersion, we fix the following notation:
\begin{enumerate}
\item $K$ is the Killing vector field on $M$ induced by the $\bbS^1$-action.
\item $l \coloneqq \vert K \vert$.
\item $i \omega : TM \rightarrow i \bbR$ is the unique connection form such that $\ker (\omega)$ is orthogonal to the fibers with respect to $g$.
\item $F \coloneqq d \omega$ is the curvature form of $\omega$ 
\item $\mathcal{F}$ is the unique two form on $N$ such that $f^{\ast} \mathcal{F} = l F$.
\end{enumerate}

Since we want to use O'Neill's formulas we recall the two fundamental tensors for Riemannian submersion: The $T$-tensor which is related to the second fundamental form of the fibers, and the $A$-tensor which is related to the integrability of the horizontal distributions. Straightforward calculations show the relations between these two tensors and the data of the $\bbS^1$-principal orbifold bundle. Since we will most of the time refer to objects on the limit space it is convenient to consider collapsing sequences in $\mfdspaceCnn$ such that the limit space is $n$-dimensional.

\begin{lem}\label{Tensors}
Let $f:(M,g) \rightarrow (N,h)$ be an $\bbS^1$-orbifold bundle such that $f$ is a Riemannian submersion. For a local orthonormal frame $(e_0, e_1 \ldots, e_n)$ on $M$ where $e_0$ is vertical and $e_1, \ldots, e_n$ horizontal, there are the following identities
\begin{align*}
T(e_0, e_0) &= - \frac{1}{l} \grad(l) , \\
T(e_0, e_i) &= - \frac{1}{l} e_i(l) e_0 , \\
A(e_i, e_0) &= \frac{l}{2} \sum_{j=1}^n F(e_i,e_j)e_j , \\
A(e_i, e_j) &= - \frac{l}{2} F(e_i, e_j) e_0 .
\end{align*}
\end{lem}

Furthermore, outgoing from Theorem \ref{CollapseSpace} we obtain a uniform bounds on these tensors.

\begin{cor}\label{SubmersionBounds}
Let $(\Ma, \ga)_{a \in \bbN}$ be sequence in $\mfdspaceCnn$ collapsing to an $n$-dimensional Riemmanian orbifold $(N,h)$. Suppose further, that for each $a \in \bbN$ there is a Riemannian submersion $\fa : (\Ma, \ga) \rightarrow (N, \ha)$. Then there are positive constants $C_A \coloneqq C_A(n,d,C)$ and $C_T \coloneqq C_T(n,d,C)$ such that $\vert \Aa \vert \leq C_A$ and $\vert \Ta \vert \leq C_T$ for all $a \in \bbN$.
\end{cor}

\begin{proof}
By Theorem \ref{CollapseSpace} the sectional curvature of $(N, \ha)$ is uniformly bounded from above by a constant $K(n,d,C)$. Thus, the uniform bound on the $A$-tensor follows directly from O'Neill's formula. Again using the uniform bound on the sectional curvature of $(N, \ha)$ the bound on the second fundamental form follows with \cite[Theorem 4.1]{RongSummary}.
\end{proof}

Combining Lemma \ref{Tensors} and Corollary \ref{SubmersionBounds} we study the following limits.

\begin{lem}\label{COneBoundedATensor}
Let $(\Ma , \ga)_{a \in \bbN}$ be a collapsing sequence of orientable manifolds in $\mfdspaceCnn$ converging to an orientable Riemannian orbifold $(N,h)$. Further suppose that for each $a \in \bbN$ there is a Riemannian submersion $f_a: (\Ma, \ga) \rightarrow (N,\ha)$. Then there is a subsequence of $(\Ma, \ga)$ such that the corresponding sequence $ (\cFa)_{s \in \bbN}$ is uniformly bounded in $C^1(N)$.
\end{lem}

\begin{proof}
By Lemma \ref{Tensors} and Corollary \ref{SubmersionBounds} it is enough to show that $\Vert \nabla \cFa \Vert_{C^0}$ is uniformly bounded. For this purpose, let $(\xi_1, \ldots , \xi_n)$ be a local orthonormal frame, parallel in $p \in N$. Denote by $(e_1, \ldots, e_n)$ the horizontal lift of this orthonormal frame and by $e_0 \coloneqq \frac{K}{l}$ the vertical unit vector. We rewrite the pointwise norm at $p$ as follows:
\begin{align*}
\vert \nabla \cFa \vert^2 &= \frac{1}{2} \sum_{i,j,k} \vert (\nabla_{\xi_i} \cFa )(\xi_j, \xi_k ) \vert^2\\
&= \frac{1}{2} \sum_{i,j,k} \vert \xi_i \big( \cFa(\xi_j, \xi_k) \big)\vert^2 \\
&= 2 \sum_{i,j,k >0} \vert e_i \big( \langle \Aa (e_j, e_k) , e_0 \rangle \big) \vert^2 \\
&= 2 \sum_{i,j,k >0} \vert \langle \nabla_{e_i} (\Aa(e_j, e_k)), e_0 \rangle + \langle \Aa(e_j, e_k), \Gamma_{i0}^0 e_0 \rangle \vert^2 \\
&= 2 \sum_{i,j,k >0} \vert \langle (\nabla_{e_i} \Aa)(e_j, e_k) , e_0 \rangle + \langle \Aa( (\nabla_{e_i} e_j)^{\mathcal{H}},e_k), e_0  \rangle \\
&\qquad \qquad \ \ + \langle \Aa(e_j, (\nabla_{e_i} e_k )^{\mathcal{H}}), e_0 \rangle \vert^2 \\
&= 2 \sum_{i,j,k >0} \vert \langle (\nabla_{e_i} \Aa)(e_j, e_k) , e_0 \rangle \vert^2
\end{align*}
Here we used that $\Gamma_{i0}^0 = 0$ and that $( \nabla_{e_i} e_j )^{\mathcal{H}} = \widetilde{\nabla_{\xi_i} \xi_j} = 0$ for all $i,j \neq 0$. By O'Neill's formula \cite[9.28 e)]{Besse},
\begin{align*}
2& \sum_{i,j,k >0} \vert \langle (\nabla_{e_i} \Aa)(e_j, e_k) , e_0 \rangle \vert^2\\
&\quad  = 2 \sum_{i,j,k >0} \Bigg( \vert \langle \Ra(e_k, e_j)e_i, e_0 \rangle - \langle \Aa(e_j, e_k), \Ta(e_0,e_i) \rangle \\
&\quad \qquad \qquad \quad \  + \langle \Aa (e_k, e_i), \Ta (e_0, e_j) \rangle + \langle \Aa (e_i, e_j), \Ta (e_0, e_k) \rangle \vert^2 \Bigg) \\
&\quad \leq 2 n^3 (C_R + 3C_A C_T)^2
\end{align*}
Here $C_R$ is a positive constant such that $\vert \Ra \vert \leq C_R$ following from the assumption that $\vert \sec(\Ma) \vert \leq 1$. 
\end{proof}

Together with Theorem \ref{invariantMetric} and the compact embedding $C^1 \hookrightarrow C^{0,\alpha}$ this lemma implies

\begin{cor}\label{ConvergenceATensor}
Let $(\Ma, \ga)_{a \in \bbN}$ be a collapsing sequence of orientable manifolds in $\mfdspaceCnn$ converging to an orientable Riemannian orbifold $(N,h)$ as in Lemma \ref{COneBoundedATensor}.  Then there is a subsequence of $(\Ma, \ga)$ such that for $(\la \Fa)_{a \in \bbN}= (f_a^{\ast}(\cFa))_{t \in \bbN}$  the sequence $(\cFa)_{a \in \bbN}$ on $N$ converges in $C^{0, \alpha}(N)$ for any $\alpha \in [0,1)$.
\end{cor}

If $(\Ma, \ga)_{a \in \bbN}$ is a collapsing sequences of orientable manifolds in $\mfdspaceCnn$ converging to a nonorientable Riemannian orbifold $N$ such that for each $a \in \bbN$ there is a Riemannian submersion $\fa: (\Ma, \ga) \rightarrow (N, \ha)$,the vertical distribution $\mathcal{V}_a$ is the pullback of the determinant bundle $\mathcal{K}$ over $N$. Considering the $A$-tensor of the Riemannian submersion $\fa: \Ma \rightarrow N$ as a map $A_a:  \mathcal{H} \times \mathcal{H} \rightarrow \mathcal{V}$ there is a two-form $\cFa$ on $N$ with values in $\mathcal{K}$ such that 
\begin{align*}
\fa^{\ast}\cFa = - 2 A_a,
\end{align*}
compare with Lemma \ref{Tensors}. We observe that the bounds given from Lemma \ref{COneBoundedATensor} also carry over to the case of $N$ being nonorientable and therefore we similarly obtain

\begin{cor}\label{NonorientableConvergenceATensor}
Let $(\Ma, \ga)_{a \in \bbN}$ be a collapsing sequence of orientable manifolds in $\mfdspaceCnn$ converging to a nonorientable Riemannian orbifold $(N,h)$ such that for all $a \in \bbN$ there is a Riemannian submersion $\fa: (\Ma, \ga) \rightarrow (N, \ha)$.  Then there is a subsequence of $(\Ma, \ga)$ such that the two-forms $\cFa \in \Omega^2(N, \mathcal{K})$ satisfying $\fa^{\ast} \cFa = - 2A_a$ converges in $C^1(N)$.
\end{cor}

\section{Spin Structures on $\bbS^1$-bundles}

In the case of a collapsing sequence of spin manifolds $(\Ma, \ga)_{a \in \bbN}$ in $\mfdspaceCnn$ with limit space $N$ we have to deal with $\bbS^1$-orbifold bundles $\Ma \rightarrow N$. Similar to \cite{Ammann} and \cite{MoroianuDiss} we need to distinguish between two types of spin structures on the manifolds $\Ma$: The projectable and the nonprojectable spin structures. 

\begin{defi}
Let $M \rightarrow N$ be an $\bbS^1$-orbifold bundle with $M$ being spin. Then the spin structure of $M$ is called \textit{projectable}, if every local $\bbS^1$-action lifts to the topological spin structure.
\end{defi}

Projectable spin structures and projectable spinors where studied for $G$-principal bundles with compact Lie group $G$ in \cite{MoroianuDiss}. As in the general case $\bbS^1$ does not act by isometries,  we replaced the spin structure by the larger so-called topological spin structure $\phi : P_{\widetilde{\GL}_+}M \rightarrow P_{Gl_+}M$, in the above definition. Here $P_{\GL_+}M$ is the $\GL(n)$-principal bundle consisting of all oriented frames and $P_{\widetilde{\GL}_+}M$ is a double cover of $P_{\GL_+}M$ which is compatible with the group double cover $\widetilde{\GL}_+(n) \rightarrow \GL_+(n)$.
Nevertheless, we have to verify this definition in the setting of $\bbS^1$-orbifold bundles. Therefore, we first generalize it to the case $\bbS^1$-principal orbifold bundles and then to the case of $N$ being nonorientable.

The first definition of spin orbifolds, to the author's knowledge, appeared in \cite{Dong}.
\begin{defi}\label{spinOrbifold}
An oriented Riemannian orbifold $(N,h)$ is spin if there exists a two-sheeted covering of $P_{\SO}(N)$ such that for any orbifold chart $\left( \tilde{U} \rightarrow \faktor{\tilde{U}}{G_U} \cong U \subset N \right)$ there exists a principal $\Spin(n)$-bundle $P_{\Spin}(\tilde{U})$ on $\tilde{U}$ such that $P_{\Spin}(N)_{\vert U} \rightarrow P_{\SO}(N)_{\vert U}$ is induced by $P_{\Spin}(\tilde{U}) \rightarrow P_{\SO}(\tilde{U})$.
\end{defi}
Thus, the spin structure on an orbifold can be defined as a locally $G_p$-invariant spin structure on the smooth covering around $p \in N$. Here $G_p$ is the stabilizer group of the Riemannian orbifold $(N,h)$ at $p$. This requires a lift of the group $G_p$ of isometries to the spin bundle.
\begin{defi}
A singular point $p \in N$ is said to be \textit{spin} if there is a lift $\widetilde{G}_p$ of $G_p \subset \SO(n)$ which projects isomorphically onto $G_p$ via the canonical projection from $\Spin(n)$ to $\SO(n)$.
\end{defi}
Henceforth, a spin orbifold is an orbifold with a fixed spin structure.

Let $f: M \rightarrow N$ be an $\bbS^1$-principal orbifold bundle. We extend the notion of a projectable spin structure canonical to $\bbS^1$-principal orbifold bundles, i.e.\ the spin structure on $M$ is \textit{projectable} if the $\bbS^1$ action lifts to the (topological) spin structure.

\begin{prop}\label{projectableSpinOrbifold}
Let $f: M \rightarrow N$ be an $\bbS^1$-principal orbifold bundle. If $M$ is a spin orbifold with projectable spin structure there is an induced spin structure on $N$. On the other hand, if $N$ is a spin orbifold it induces a projectable spin structure on $M$.
\end{prop}

\begin{proof}
As all metric spin structures are isomorphic to each other, we can assume without loss of generality, that $f: M \rightarrow N$ is a Riemannian orbifold submersion and $\bbS^1$ acts by isometries. In the following, the proof is a locally equivariant version of the construction given in \cite[Chapter 1]{MoroianuDiss}.

For $p \in N$ we consider a local trivialization $U$ around $p$. Then the local situation looks as follows:
\begin{align*}
 \begin{xy}
 	\xymatrix@R-2em@C-0.5em{
 	\tilde{U} \times \bbS^1 \ar[rd] \ar[dd] &  \\
 	 & \faktor{(\tilde{U} \times \bbS^1)}{G_U} \cong f^{-1}(U) \ar[dd] \\
	\tilde{U} \ar[rd] & \\
	& \faktor{\tilde{U}}{G_U} \cong U	
 	}
 \end{xy}
\end{align*}
Hence, the spin structure on $f^{-1}(U)$ is $G_U$ invariant. In particular, the group $G_U$ of isometries lift to the spin structure on the smooth covering $\tilde{U} \times \bbS^1$. Observe, that as $G_p$ is a subgroup of $U$ it also lifts to the spin structure.

If the spin structure on $M$ is projectable i.e.\ $\bbS^1$-equivariant, the spin structure on $\tilde{U} \times \bbS^1$ is $G_U \times \bbS^1$ invariant. Hence, it follows by a standard construction that the spin structure on $M$ induces a spin structure on $N$, see the commutative diagram above.

On the other hand, if $N$ is a spin orbifold it follows that the spin structure on $M$ induced by the pull back of the spin structure on $N$ has to be $\bbS^1$-equivariant, i.e.\ projectable.
\end{proof}

Next we need to extend the notion of a projectable spin structure to the case of $\bbS^1$-orbifold bundles $f:M  \rightarrow N$ where $M$ is spin and $N$ is nonorientable. As $N$ is nonorientable it does admit an orthonormal frame bundle $P_{\mathrm{O}} N$ but not an oriented orthonormal frame bundle. Therefore, we consider pin structures which generalizes spin structures. In the following we roughly sketch the definitions and properties of pin structures. For further details we refer to \cite{Trautman} and \cite[Appendix A.1]{GilkeyPin}

There are two inequivalent double coverings of $\mathrm{O}(n)$ by the groups $\Pin^{\pm}(n)$ which coincide on their preimage of $\SO(n)$.

\begin{defi}
A manifold $(M,g)$ is $\text{pin}^{\pm}$ if it admits a $\text{pin}^{\pm}$  structure, i.e.\ there is a $\Pin^{\pm}(n)$-principal bundle $P_{\Pin^{\pm}(n)} M$ such that it is a double covering of the orthonormal frame bundle $P_{\mathrm{O}} M$ compatible with the double covering $\Pin^{\pm}(n) \rightarrow \mathrm{O}(n)$.
\end{defi}

While a $\text{pin}^+$ structure is equivalent to a spin structure there are nonorientable manifolds carrying a \pin - structure, e.g.\ $\bbR P^2$. $\text{Pin}^{\pm}$-structures on orbifolds are similarly defined as spin orbifolds, see Definition \ref{spinOrbifold}.

Using this definition we derive
\begin{prop}
Let $f:M \rightarrow N$ be an $\bbS^1$-orbifold bundle where $N$ is an unorientable Riemannian orbifold. Then any projectable spin structure on $M$ induces a \pin structure on $N$. In contrast, if $N$ is \pin and $M$ orientable, then there is an induced spin structure on $M$.
\end{prop}
\begin{proof}
The proof is similar to the proof of Proposition \ref{projectableSpinOrbifold} by identifying $f^{\ast} P_{\mathrm{O}} N$ with a subbundle of $P_{\SO}M$ via the embedding
\begin{align*}
\mathrm{O}(n) &\hookrightarrow \SO(n+1) \\
A & \mapsto \begin{pmatrix}
\det(A) & 0 \\
0 & A
\end{pmatrix}.
\end{align*}
\end{proof}


Now let $f: M \rightarrow N$ be an $\bbS^1$-principal orbifold bundle such that the spin structure on $M$ is nonprojectable. As before, we assume without loss of generality that $\bbS^1$ acts by isometries. In particular, $f$ is a Riemannian submersion.

As the spin structure of $M$ is nonprojectable the $\bbS^1$-action does not lift to $P_{\Spin}(M)$. Nevertheless, the double-cover of $\bbS^1$ acts on $P_{\Spin}(M)$. 

A nonprojectable spin structure on $N$ does not imply that $N$ is not spin. If $N$ is spin, there exists a group homomorphism $\psi: \pi_1(M) \rightarrow \bbZ_2$ such that the composition is $\pi_1(\bbS^1) \hookrightarrow \pi_1(M) \rightarrow \bbZ_2$ is surjective. Then we can twist the spin structure on $M$ with $\psi$ to obtain a projectable spin structure, i.e.\ N is spin if and only if  $M \rightarrow N$ has a square root as $\bbS^1$-bundle, cf.\ \cite[Chapter 7.3]{AmmannDiss}.

Even, if we can not determine if $N$ is spin or not, we still have an induced structure on $N$. Here we extend the proof of \cite[Section 4]{Ammann}

\begin{lem}
Let $f:M \rightarrow N$ be an $\bbS^1$-principal orbifold bundle. If $M$ is a spin orbifold with nonprojectable spin structure, there is an induced $\text{spin}^{\bbC}$-structure on $N$. 
\end{lem}

\begin{proof}
Let $P_{\SO(n)}M$ be the $\SO(n)$-principal bundle over $M$ consisting of all positive oriented orthonormal frames whose first vector is vertical. Its preimage defines a principal $\Spin(n)$-bundle $P$. Recall, that not the $\bbS^1$-action itself but its double cover acts on $P$. This group operation together with the $\Spin(n)$-action on $P$ induces a free $\Spin^{\bbC}(n) \coloneqq \left( \Spin(n) \times_{\bbZ_2} \bbS^1 \right)$-action on $P$. Thus, $P$ defines a $\text{spin}^{\bbC}$-structure on $N$. 
\end{proof}

Conversely, if we have a fixed $\bbS^1$-principal orbifold bundle $f: M \rightarrow N$ such that $N$ has a $\text{spin}^{\bbC}$-structure, then it does not follow that $M$ is a spin manifold.

\begin{ex}
Let $M \coloneqq \bbC P^2 \times \bbS^1$ be the trivial $\bbS^1$-bundle over the complex projective space $\bbC P^2$. It is known that $\bbC P^2$ is $\text{spin}^{\bbC}$ but not spin. Thus, $M$ does not admit any spin structure.
\end{ex}

\begin{rem}
If $f:M \rightarrow N$ is an $\bbS^1$-orbifold bundle such that the spin structure on $M$ is nonprojectable and $N$ is nonorientable then this does not induce a $\text{pin}^{\bbC}$-structure on $N$, because $f:M \rightarrow N$ is not an $\bbS^1$-principal orbifold bundle, compare \cite[p.\ 312]{GilkeyPin}.
\end{rem}

\section{Induced Operators}

Consider a sequence $(\Ma, \ga)_{a \in \bbN}$ in $\mfdspaceCnn$ collapsing to a Riemmannian orbifold $(N,h)$. We assume that for each $a$ the manifold $\Ma$ is spin and that the metric $\ga$ is invariant, see Theorem \ref{invariantMetric}. Furthermore, for each $a \in \bbN$ let $\Za$ be an element of $\Hom(\Sigma \Ma , \Sigma \Ma)$, where $\Sigma \Ma$ is the spin bundle of $(\Ma, \ga)$.

The goal of this section is to study the behavior of the sequence $(\Za)_{a \in \bbN}$. We show that under appropriate condition this sequence converge to a well-defined operator $\mathcal{Z} \in \Hom (\Sigma N , \Sigma N)$ if $N$ is orientable and a well-defined operator $\mathcal{Z} \in \Hom (\Sigma^p N \otimes \mathcal{K}^{\bbC} , \Sigma^p N \otimes \mathcal{K}^{\bbC})$ if $N$ is nonorientable. Here $\Sigma^p N \otimes \mathcal{K}^{\bbC}$ is the \pin  bundle on $N$ twisted with the complexified determinant bundle $ \mathcal{K}^{\bbC}$.

To simplify notation we define
\begin{align*}
\opspace \coloneqq  \left\lbrace (M,g,Z): \ \begin{matrix}
(M,g) \in \mfdspaceCnn \ \text{and spin} \\
Z \in \Hom(\Sigma M , \Sigma M)
\end{matrix} \right\rbrace
\end{align*}

Collapsing $\bbS^1$-principal bundles of spin manifolds were discussed, under slightly different assumptions, in \cite{Ammann}. We adapt his setting to our situation.

First we consider an $\bbS^1$-principal orbifold bundle $f:(M,g) \rightarrow (N,h)$ where $f$ is a Riemannian submersion. Recall that $K$ is the Killing field on $M$ induced by the $\bbS^1$-action. If the spin structure on $M$ is projectable, then the $\bbS^1$-action lifts to an isometric action $\kappa: \bbS^1 \times \Sigma M \rightarrow \Sigma M$. As a shortcut, we denote with  $\kappa_t$ the action of the element $e^{i2\pi t} \in \bbS^1$ wherever it is considered. If the spin structure in nonprojectable, then the double cover of $\bbS^1$ acts on $\Sigma M$. We denote this action also with $\kappa$. We define the Lie-derivative of a spinor $\varphi$ in the direction of $K$ as follows:
\begin{align*}
\mathcal{L}_{K}(\varphi) (x) \coloneqq \left. \frac{\mathrm{d}}{\mathrm{d}s} \right\vert_{s= 0} \kappa_{-s}(\varphi (\kappa_s (x) ) ).
\end{align*}
By construction, $\mathcal{L}_K$ is the differential of the $\bbS^1$-action on $L^2(\Sigma M)$. Thus, it has the eigenvalues $ik$ where $k \in \bbZ$ if the spin structure on $M$ is projectable and $k \in (\bbZ + \frac{1}{2})$ if the spin structure on $M$ is nonprojectable. Denote with $V_k$ the eigenspace of $\mathcal{L}_K$ to the eigenvalue $ik$.  Hence, $L^2(\Sigma M)$ decomposes as
\begin{align*}
L^2(\Sigma M) = \bigoplus_{k} V_k.
\end{align*}

\begin{rem}
As $\kappa$ acts on $\Sigma M$ by isometries, it commutes with the Dirac Operator $D^M$. Therefore, $\mathcal{L}_K$ and $D^M$ are simultaneously diagonalizable, i.e.\ for any eigenspinor $\varphi$ of $D^M$ there is a $k \in \bbZ$ (resp.\ $k \in (\bbZ + \frac{1}{2})$ ) such that $\varphi \in V_k$.
\end{rem}

\begin{lem}[ \cite{Ammann}]\label{verticalRelation}
For any $k \in \bbZ$, resp.\ $k \in (\bbZ + \frac{1}{2})$, and any spinor $\varphi \in V_k$,
\begin{align*}
\nabla_{K} \varphi - \mathcal{L}_K = \frac{l^2}{4} \gamma(F) \varphi - \frac{1}{2} \gamma(K) \gamma \left(\frac{\grad(l)}{l} \right) \varphi 
\end{align*}
Here, the Clifford multiplication with a two-form is defined as
\begin{align*}
\gamma(F) \varphi \coloneqq \sum_{i \leq j} F(e_i, e_j) \gamma(e_i) \gamma(e_j) \varphi.
\end{align*}
\end{lem}

Recall that 
\begin{align*}
\Sigma_{n+1} \simeq 
\begin{cases}
\Sigma_n  &\text{if $n$ is even,}\\
\Sigma_n \oplus \Sigma_n &\text{if $n$ is odd.}
\end{cases}
\end{align*}
Thus, if $n$ is odd, we consider $\nu_n = f^{\ast} \omega_n^{\bbC} = i^{\left[ \frac{n+1}{2} \right]} \gamma(e_1) \ldots \gamma(e_n)$, the pullback of the complex volume form of $\Sigma N$. As the square of $\nu_n$ is the identity the map $\nu_n: \Sigma_{n+1} \rightarrow \Sigma_{n+1}$ has the eigenvalues $\pm 1$. Thus, we obtain the following splitting
\begin{align*}
\Sigma_{n+1} = \Sigma_n^{+} \oplus \Sigma_n^{-}
\end{align*}
into the corresponding eigenspaces.  Then, $i \gamma(e_0): \Sigma_{n}^{\pm} \rightarrow \Sigma_n^{\mp}$ defines an isometry. This action anti commutes with Clifford multiplication of horizontal vector fields.

\begin{rem}
As $n+1$ is even, there is a natural splitting $\Sigma_{n+1} = \Sigma^+_{n+1} \oplus \Sigma_{n+1}^{-}$ into the $\pm 1$-eigenspaces of the complex volume element $\omega^{\bbC}_{n+1} = i^{\left[ \frac{(n+1)+1}{2} \right]} \gamma(e_0) \gamma(e_1) \ldots \gamma(e_n)$.  This is a \textit{different} splitting as $\nu_n$ and $\omega^{\bbC}_{n+1}$ do not commute with each other.
\end{rem}

Set $L \coloneqq M \times_{\bbS^1} \bbC$. Ammann constructed in \cite[Lemma 3.2]{Ammann} for each $k \in \bbZ$ (resp.\ $k \in (\bbZ + \frac{1}{2})$) an isometry
\begin{align*}
Q_k : \begin{cases}
L^2( \Sigma N \otimes L^{-k}) \rightarrow V_k, &\text{if $n$ is even} , \\
L^2( (\Sigma^+N \oplus \Sigma ^- N) \otimes L^{-k}) \rightarrow V_k, &\text{if $n$ is odd}.
\end{cases}
\end{align*}
In the case of nonprojectable spin structures, the tensor product of  the bundles $\Sigma N \otimes L^{-k}$ exist, the separate bundles itself are not necessarily defined globally.

The map $Q_k$ behaves well with Clifford multiplication. For a vector field $X$ on $N$, let $\tilde{X}$ denote its horizontal lift. For any spinor $\phi$,
\begin{align*}
\gamma(\tilde{X}) Q_k(\phi) = 
\begin{cases}
Q_k (\gamma(X) \phi) &\text{if $n$ is even}, \\
Q_k ( \gamma(X) \phi^+ \oplus -\gamma(X) \phi^-) &\text{if $n$ is odd},
\end{cases}
\end{align*}
and for the vertical unit vector field $V$ we have
\begin{align*}
i\gamma(V) Q_k (\phi)= 
\begin{cases}
Q_k ( \omega_n^{\bbC} \phi) &\text{if $n$ is even}, \\
Q_k ( \phi^- \oplus  \phi^+) &\text{if $n$ is odd},
\end{cases}
\end{align*}
where $\omega_n^{\bbC} \coloneqq i^{\left[ \frac{n}{2} \right]} \gamma(\xi^k_1) \ldots \gamma( \xi^k_n)$ is the complex volume element of $\Sigma N \otimes L^{-k}$ which is defined via a local orthonormal frame $(\xi^k_1, \ldots \xi^k_n)$.
\newline

Recall that we also have to consider the situation of $\bbS^1$-orbifold bundles $f:M \rightarrow N$, with $M$ spin and $N$ nonorientable. The canonical representations for the Clifford algebra $\mathbf{C}l (n)$ can be also restricted to $\Pin^-(n)$. We call the associated vector bundle $\Sigma^P N$ the \pin bundle of $N$. Recall the embedding
\begin{align*}
\iota : \mathrm{O}(n) &\hookrightarrow \SO(n+1) \\
A & \mapsto \begin{pmatrix}
\det(A) & 0 \\
0 &A
\end{pmatrix}.
\end{align*}
Let $\tilde{\iota}: \Pin^-(n) \hookrightarrow \Spin(n+1)$ be the lift of this embedding to the double cover. Assuming $n$ to be even, it follows that
\begin{align*}
\Sigma M &= P_{\Spin}M \times_{\rho_{n+1}} \Sigma_{n+1} \\
&\cong ( f^{\ast} P_{\mathrm{O}}N \times_{\tilde{\iota}} \Spin(n+1) ) \times_{\rho_{n+1}} (\Sigma_n \otimes \bbC) \\
&= (f^{\ast} P_{\mathrm{O}}N ) \otimes (\mathcal{K} \otimes_{\bbR} \bbC)\\
&= (f^{\ast} P_{\mathrm{O}}N ) \otimes \mathcal{K}^{\bbC},
\end{align*}
where $\mathcal{K}$ is the determinant bundle of $N$. Similar we obtain for $n$ odd
\begin{align*}
\Sigma M \cong ( \Sigma^{P+} N \oplus \Sigma^{P-}N ) \otimes \mathcal{K}^{\bbC},
\end{align*}
where the splitting is analogous to the spin case.

For our purpose it is enough to consider the case of $M$ carrying a projectable spin structure inducing a \pin - structure on $N$. Let $V_0$ denote the space of $\bbS^1$-invariant subspace of $L^2(\Sigma M)$, i.e.\ those spinors $\phi$ such that $\mathcal{L}_K \phi = 0$ for any local Killing field $K$ induced by the local $\bbS^1$-actions on $M$. Following the lines of \cite[Lemma 3.2]{Ammann} we find an isometry
\begin{align*}
Q^P_0 : \begin{cases}
L^2( \Sigma^P N \otimes \mathcal{K}^{\bbC}) \rightarrow V_0, &\text{if $n$ is even} , \\
L^2( (\Sigma^{P+}N \oplus \Sigma^{P-} N) \otimes \mathcal{K}^{\bbC}) \rightarrow V_0, &\text{if $n$ is odd}.
\end{cases}
\end{align*}
As in the spin case, $Q^P_0$ behaves well with Clifford multiplication. For a vector field $X$ on $N$, its horizontal lift $\tilde{N}$and any spinor $\phi$ we have
\begin{align*}
\gamma(\tilde{X}) Q_0^P(\phi \otimes s) = 
\begin{cases}
Q_0^P \big( (\gamma(X) \phi) \otimes s \big) &\text{if $n$ is even}, \\
Q_0^P \big( ( \gamma(X) \phi^+ \oplus -\gamma(X) \phi^-)\otimes s \big) &\text{if $n$ is odd},
\end{cases}
\end{align*}
and for the vertical unit vector field $V$ we have
\begin{align*}
i\gamma(V) Q_0^P (\phi \otimes s)= 
\begin{cases}
Q_0^P \big( (\omega_n^{\bbC} \phi) \otimes s \big) &\text{if $n$ is even}, \\
Q_0^P \big( ( \phi^- \oplus  \phi^+) \otimes s \big) &\text{if $n$ is odd},
\end{cases}.
\end{align*}

Since we want to consider limit operators acting on the spinors of $N$, it is convenient to assume that the spin structure on $M$ is projectable. To simplify notation we only carry out the case of an $\bbS^1$-principal orbifold bundle $f:M \rightarrow N$. The statements and modifications for the remaining case are obvious. Hence, let $f:M \rightarrow N$ be an $\bbS^1$-principal orbifold bundle such that $M$ has a projectable spin structure. In that case, $0$ is an eigenvalue of $\mathcal{L}_K$ and we, therefore, have the isometry
\begin{align*}
Q_0 : \begin{cases}
L^2( \Sigma N) \rightarrow V_0, &\text{if $n$ is even} , \\
L^2( (\Sigma^+N \oplus \Sigma ^- N)) \rightarrow V_0, &\text{if $n$ is odd}.
\end{cases}
\end{align*}

Let $Z \in \Hom(\Sigma M, \Sigma M)$: By the above discussion, $Z_{\vert V_0}$ can only be identified, via $Q_0$ with an operate $\mathcal{Z}$ on $N$ if $Z(V_0) \subset V_0$ or equivalently \ $\mathcal{L}_K( Z) = 0$. We call such an operator \textit{projectable}.

\begin{defi}
Let $\eta: M \rightarrow N$ be an $\bbS^1$-principal orbifold bundle such that $\bbS^1$ acts by isometries and $M$ has a projectable spin structure. For  $Z \in \text{Hom} (\Sigma M , \Sigma M)$ acting on spinors, we define the \textit{associated invariant operator} as
\begin{align*}
\tilde{Z} (\varphi) \coloneqq \int_0^{1} \kappa_{-t} ( Z ( \kappa_t \varphi) ) dt,
\end{align*}
where $\kappa$ is the induced $\bbS^1$-action on $\Sigma M$.
\end{defi}

\begin{lem}
For any $Z \in \text{Hom} (\Sigma M , \Sigma M)$  the operator $\tilde{Z}$ induces a well-defined operator 
\begin{align*}
\tilde{Z}: V_0 \rightarrow V_0.
\end{align*}
\end{lem}

\begin{proof}
Recall that $V_0 \coloneqq \lbrace \varphi \in L^2(\Sigma M) \vert \mathcal{L}_K (\varphi) =0 \rbrace$. Therefore, we need to show that $\mathcal{L}_K (\tilde{Z} \varphi ) = 0$ for any $\varphi \in V_0$. Let  $\varphi \in V_0$. Then $\kappa_s (\varphi(x)) = \varphi(\kappa_s x)$ for all $s \in [0,1]$ and
\begin{align*}
\mathcal{L}_K (\tilde{Z}(\varphi) ) (x) &= \left. \frac{d}{ds} \right\vert_{s=0} \kappa_{-s} (\tilde{Z} (\varphi) ) (\kappa_s x) \\
&= \left. \frac{d}{ds} \right\vert_{s=0}  \int_0^{1} \kappa_{-s-t}(Z (\kappa_t( \varphi (\kappa_{s} x ))) dt \\
&= \left. \frac{d}{ds} \right\vert_{s=0} \int_0^{1} \kappa_{-s-t}(Z (( \varphi (\kappa_{t+s} x ))) dt \\
&= \left. \frac{d}{ds} \right\vert_{s=0}   \int_0^{1} \kappa_{-t}(Z (( \varphi (\kappa_{t} x ))) dt = 0.
\end{align*}
\end{proof}

If a sequence $(\Za)_{a \in \bbN}$ associated to a collapsing sequence $(\Ma, \ga)_{a \in \bbN}$ in $\mfdspaceCnn$ should converge to a projectable operator we need to ensure that $\Vert \Za - \tZa \Vert_{\infty}$ goes to $0$ as $a$ tends to infinity. This is the content of the next proposition.

\begin{prop}\label{ConvergenceAssOp}
Let $(\Ma, \ga, \Za)_{a \in \bbN}$ be a collapsing sequence in $\opspace$ such that the spin structure on $\Ma$ is projectable for all $a \in \bbN$. Then $\lim_{a \rightarrow \infty}\Vert \Za \Vert_{W^{1, \infty}} \inj(\Ma) = 0$ implies
\begin{align*}
\lim_{a \rightarrow \infty} \Vert \left. \tZa \right._{\vert V_0(a)} - \left.\Za\right._{\vert V_0(a)} \Vert_{L^{\infty}} = 0.
\end{align*}
\end{prop}
\begin{proof}

First we note that by Theorem \ref{invariantMetric} we can switch to invariant metrics $\tga$ such that $\lim_{a \rightarrow \infty} \Vert \ga - \tga \Vert_{C^1} = 0$. Furthermore, the spinor bundles $\Sigma \Ma$ and $\widetilde{\Sigma \Ma}$ with respect to $\ga$ resp.\ $\tga$ are isomorphic. Hence, we can pull back $\Za$ to an operator in $\Hom(\widetilde{\Sigma \Ma},\widetilde{\Sigma \Ma})$. Therefore, we can assume without loss of generality, the metrics $\ga$ to be invariant.

Let $\phia \in V_0(a)$ with $\Vert \phia \Vert_{L^{\infty}} = 1$. Then,
\begin{align*}
\Vert (\tZa - \Za ) \phia \Vert_{L^{\infty}} & = \left\Vert  \int_0^{1}\kappa_t (\Za (\kappa_t \phia ) ) - \Za (\phia) dt \right\Vert_{L^{\infty}}\\
& = \left\Vert  \int_0^{1} \int_0^t \kappa_{-s} \mathcal{L}_{\Ka} (\Za (\phia) ) ds \, dt \right\Vert_{L^{\infty}} \\
& \leq \frac{1}{2} \Vert \mathcal{L}_{\Ka} ( \Za ( \phia ) ) \Vert_{L^{\infty}}
\end{align*}
Applying Lemma \ref{verticalRelation} and Corollary \ref{SubmersionBounds} we conclude
\begin{align*}
 \Vert \mathcal{L}_{\Ka} (\Za (\phia)) \Vert_{L^{\infty}} & \leq   \Vert \nabla_{\Ka} ( \Za(\phia) ) \Vert_{L^{\infty}} + \Vert \Ka \Vert_{L^{\infty}} (C_T + C_A) \\
 &\leq \Vert \Ka \Vert_{L^{\infty}}  \Vert \nabla \Za \Vert_{L^{\infty}} +  \Vert \Za \Vert_{L^{\infty}} \Vert \nabla_{\Ka} \phia \Vert_{L^{\infty}} \\
 & \ \ \ + \Vert \Ka \Vert_{L^{\infty}} (C_T + C_A) \\
 &\leq \Vert \Ka \Vert_{L^{\infty}}  \Vert \nabla \Za \Vert_{L^{\infty}} + \Vert \Ka \Vert_{L^{\infty}}  \Vert\Za \Vert_{L^{\infty}} (C_T + C_A)\\
 & \ \ \ + \Vert \Ka \Vert_{L^{\infty}} (C_T + C_A).
\end{align*}

First we note, that $\inj(\Ma) = \inj^{\Ma}(x_a)$ for some $x_a \in \Ma$. As the second fundamental form of the fibers of $f_a : \Ma \rightarrow N$ is uniformly bounded by $C_T$, see Corollary \ref{SubmersionBounds}, there is a positive constant $C_1(d,C_T)$ such that
\begin{align*}
\Vert \Ka \Vert_{L^{\infty}} \leq  C_1(d,C_T) \vert K_{x_a} \vert = C_1(d,C_T) \frac{1}{\pi} \inj(F_{p_a})
\end{align*}
where $F_{p_a} \coloneqq f_a^{-1}(p_a ) \cong \bbS^1$ is the fiber over $p_a \coloneqq f_a(x_a)$. By combining Corollary \ref{SubmersionBounds} and \cite[Proposition 1.4]{Roos}, there is a further constant $C_2$ such that \begin{align*}
\inj(F_{p_a}) \leq C_2 \inj^{\Ma}(x_a) = C_2 \inj(\Ma).
\end{align*} 
Thus, the claim follows.
\end{proof}

\section{Dirac Operators with potential}

In this section we describe the behavior of the spectrum of Dirac operators with symmetric $W^{1, \infty}$-potential. For any collapsing sequence of spin manifolds in $\mfdspaceCnn$ there is a subsequence either consisting only of spin manifolds with nonprojectable spin structure or consisting only of those with projectable spin structure such that they all induce the same spin structure, resp.\ \pin structure on the limit space $N$. Therefore we consider these two cases separately.

We will show that in the case of nonprojectable spin structures all eigenvalues  diverge, whereas in the case of projectable spin structures only a part of the spectrum diverge while the other part converges to the spectrum of a Dirac operator $D$ with $W^{1, \infty}$-potential on the limit space $N$. If $N$ is odd, then $D$ is the classical Atiyah-Singer Dirac operator $D^N$ on the spinor bundle $\Sigma N$ if $n$ is even and if $n$ is odd then $D = D^N\oplus -D^N$ is the  Dirac operator on $\Sigma^+N \oplus \Sigma^- N$. If $N$ is nonorientable then $D$ is the twisted Dirac operator $\tilde{D}^N$ on the twisted \pin  \ bundle $\Sigma^P N \otimes \mathcal{K}^{\bbC}$ if $n$ is even and  $D = \tilde{D}^N\oplus -\tilde{D}^N$ is the twisted Dirac operator on $(\Sigma^{P+} \oplus \Sigma^{P-}) \otimes \mathcal{K}^{\bbC}$ if $n$ is odd, where $\mathcal{K}^{\bbC}$ is the complexified determinant bundle of $N$.

 \subsection{The case of nonprojectable spin structures}

First we deal with sequences $(\Ma, \ga, \Za)_{a \in \bbN}$ in $\opspace$, where the spin structure on $\Ma$ is nonprojectable, collapsing to a Riemannian orbifold $(N,h)$. Recall that after passing to invariant metrics (see Theorem \ref{invariantMetric} ) and to the orientation covering, the space of $L^2$-spinors decomposes as
\begin{align*}
L^2(\Sigma \Ma ) = \bigoplus_{k \in \left( \bbZ + \frac{1}{2} \right)} V_k(a),
\end{align*}
where $V_k(a)$ is the eigenspace of the Lie derivative $\mathcal{L}_{K_a}$ along the fibers of the $\bbS^1$-bundle $f_a: \Ma \rightarrow N$ with respect to the eigenvalue $ik$. In this setting, $0$ is not an eigenvalue. Thus, there are no spinors which are invariant under the $\bbS^1$-action. This can be interpret as an indication why the eigenvalues of $\Da + \Za$ should diverge in the limit.

\begin{thm}\label{nonprojectableEigenvalues}
Let $(\Ma, \ga, \Za)_{a \in \bbN}$ be a collapsing sequence in $\opspace$ such that the spin structures of $\Ma$ are nonprojectable. Suppose further that $\Za$ is symmetric and that there is a positive constant $\Lambda$ such that $\Vert \Za \Vert_{L^{\infty}} \leq \Lambda$ for all $a \in \bbN$. Then we can number the eigenvalues $\left( \lambda_{k,j}(a) \right)_{k \in \left(\bbZ + \frac{1}{2} \right), j \in \bbZ}$ of $\Da + \Za$ such that, for all $\eps > 0$ there is an $A > 0$ such that for all $a \geq A$
\begin{align*}
\vert \lambda_{k,j}(a)\vert \geq \sinh \left( \arsinh \left(\frac{k}{l_a}- \frac{1}{2}\left[ \frac{n}{2} \right]^{\frac{1}{2}} C_A  - \eps  \right) - \eps \right) - \Lambda.
\end{align*}
In particular, as $\lim_{a \rightarrow \infty} l_a = 0$ all eigenvalues diverge as $a$ tends to infinity.
\end{thm}

\begin{proof}
Considering Theorem \ref{CollapseSpace} there is an $n$-dimensional Riemannian orbifold $N$ such that a subsequence of $(\Ma, \ga, \Za)_{a \in \bbN}$ converges to $N$ in the Gromov-Hausdorff topology. In addition, it follows by Theorem \ref{invariantMetric} that $\Ma \rightarrow N$ is an $\bbS^1$-orbifold bundle with affine structure group for sufficiently large $a$. If $N$ is orientable then this is an $\bbS^1$-principal orbifold bundle and if $N$ is nonorientable we consider the pullback bundle over the orientation covering $\hat{N}$ which is then also a $\bbS^1$-principal bundle. Since nonprojectable spin structures pulls back to nonprojectable structures we can without loss of generality assume that the possible limit space $N$ is orientable.

From \cite[Chapter 5, Theorem 4.10]{Kato} it follows that
\begin{align}\label{FirstKato}
\dist (\sigma(\Da + \Za ) , \sigma(\Da) ) \leq \Vert \Za \Vert_{L^{\infty}} \leq \Lambda
\end{align}
where $\sigma(\Da + \Za)$, resp.\ $\sigma (\Da)$, denotes the spectrum of $\Da + \Za$, resp.\ $\Da$. Let $\lambda_{k,j}^D(a)$ be the eigenvalues of $\Da$. 

Apply Theorem \ref{invariantMetric} to all $a$. We obtain the invariant metrics $\tga$ satisfying
\begin{align*}
\lim_{a \rightarrow \infty}  \Vert \ga - \tga \Vert_{C^{1, \alpha}} = 0.
\end{align*}

The change of the spectra is controlled by
\begin{align}\label{Arsinh}
\vert \arsinh(\lambda_{k,j}^{\tilde{D}}(a)) - \arsinh( \lambda_{k,j}^D(a) ) \vert \leq C(a) \, , \, \lim_{a \rightarrow \infty} C(a) = 0 ,
\end{align}
as stated in \cite[Main Theorem 2]{Nowaczyk}.

Since $\bbS^1$ acts by isometries on $(\Ma, \tga)$ there is a Riemannian submersion 
\begin{align*}
f_a : \Ma \rightarrow \faktor{\Ma}{\bbS^1} \eqqcolon N.
\end{align*}
In particular, the Lie derivative $\mathcal{L}_{\Ka}$ along the fibers and the Dirac operator $\tDa$ are simultaneously diagonalizable. Therefore, we can number the eigenvalues of $\tDa$ as follows: For any fixed $k \in \left( \bbZ + \frac{1}{2} \right)$ we denote with $\lambda_{k,j}(a)$ the eigenvalues of $\left.\tDa\right._{\vert V_k(a)}$ such that
\begin{align*}
\ldots \leq \lambda_{k,-1}(a) \leq \lambda_{k,0}(a) < 0 \leq \lambda_{k,1}(a) \leq \lambda_{k,2}(a) \leq \ldots
\end{align*}

As shown in \cite{Ammann}, the Dirac operator splits as
\begin{align*}
\tDa = \frac{1}{\la} \gamma\left(\frac{\Ka}{\la}\right) \mathcal{L}_K + D^H_a - \frac{1}{4}\gamma\left(\frac{\Ka}{\la}\right) \gamma(\la \Fa),
\end{align*}
where $D^H_a$ is described by its action on the eigenspaces $V_k$ of $\mathcal{L}_{\Ka}$, namely
\begin{align*}
\left. D^H_a \right._{\vert V_k(a)} \coloneqq Q_{k,a} \circ D_{k,a} \circ Q_{k,a}^{-1}.
\end{align*}
Here $D_{k,a}$ is the twisted Dirac operator on $\Sigma N \otimes L_a^{-k}$ if $n$ is even, and on $(\Sigma^+ N \oplus \Sigma^- N) \otimes L_a^{-k}$, if $n$ is odd. 

By Lemma \ref{Tensors},  $\Vert \la \Fa \Vert_{\infty}$ is controlled by the norm of the $A$-tensor, i.e.\ by the constant $C_A$, see Corollary \ref{SubmersionBounds}. Applying \cite[Lemma 3.3]{MoroianuHerzlich} we observe that
\begin{align*}
\left\Vert \frac{1}{4}\gamma\left(\frac{K}{\la}\right) \gamma(\la \Fa) \right\Vert \leq \frac{1}{2} \left[ \frac{n}{2} \right]^{\frac{1}{2}} C_A
\end{align*} 
By \cite[Chapter 5, Theorem 4.10]{Kato} it follows that
\begin{align}\label{SecondKato}
\dist \left(\sigma (\tilde{D}_a) , \sigma \left( \frac{1}{\la} \gamma\left(\frac{K}{\la}\right) \mathcal{L}_K + D^H_a  \right) \right) \leq \frac{1}{2} \left[ \frac{n}{2} \right]^{\frac{1}{2}} C_A.
\end{align}
Let $\lambda_{k,j}^{W}(a)$ be the eigenvalues of $W \coloneqq  \frac{1}{\la} \gamma\left(\frac{K}{\la}\right) \mathcal{L}_K + D^H_a$. It was shown in \cite{Ammann} that for all $\eps > 0$ there is an $A \geq 0$ such that 
\begin{align*}
\vert \lambda_{k,j}^{W}(a)\vert \geq \frac{\vert k \vert}{\la} - \eps,
\end{align*}
for all $a \geq A$.

Going now all the reductions backwards, (i.e.\ applying all the inequalities \eqref{SecondKato}, \eqref{Arsinh}, \eqref{FirstKato} ) the claim follows.
\end{proof}

\subsection{The case of projectable spin structures}

Let $(\Ma, \ga,\Za)_{a \in \bbN}$ be a collapsing sequence in $\opspace$ with limit space $N$. We have seen in the last section that in the case of nonprojectable spin structures, all eigenvalues diverge as $a$ tends to infinity. But in the case of projectable spinors, after passing to invariant metrics (see Theorem \ref{invariantMetric} ) and to the orientation covering if necessary, 
\begin{align*}
L^2(\Sigma \Ma) = \bigoplus_{k \in \bbZ} V_k(a),
\end{align*}
where $V_k(a)$ is the eigenspace of the Lie derivative $\mathcal{L}_{K_a}$ along the fibers of the $\bbS^1$-bundle $f_a : \Ma \rightarrow N$ with respect to the eigenvalue $ik$. In particular, recall that in the case of $N$ being orientable, $V_0(a)$ is isometric to $L^2(\Sigma N)$ if $n$ is even, and isometric to $L^2(\Sigma^+ N \oplus \Sigma^- N)$ if $n$ is odd. If the limit space $N$ is nonorientable then the subspace $V_0$ is isometric to $L^2(\Sigma^ N \otimes \mathcal{})$ if $n$ is even, and isometric to $L^2(\Sigma^+ N \oplus \Sigma^- N)$ if $n$ is odd 

To summarize the result, we obtain similar lower bounds on the eigenvalues of $\Da + \Za$ as in the case of nonprojectable spin structures. But, as $k$ can be chosen to be $0$, it does not follow from this lower bound that the eigenvalues $(\lambda_{0,j}(a))_{j \in \bbZ}$ diverge. On contrary, we show that the eigenvalues $(\lambda_{0,j}(a))_{j \in \bbZ}$ converge to the eigenvalues of the Dirac operator on the respective Clifford bundle with a symmetric $W^{1, \infty}$-potential.

\begin{thm}\label{projectableEigenvalues}
Let $(\Ma, \ga , \Za)_{a \in \bbN}$ be a sequence in $\opspace$ collapsing to $(N,h)$. Suppose that the spin structures of $\Ma$ are projectable and induce the same spin structure on $N$ for all $a \in \bbN$. Suppose further that $\Za$ is symmetric and that there is a positive constant $\Lambda$ such that $\Vert \Za \Vert_{L^{\infty}} \leq \Lambda$ for all $a \in \bbN$. Then we can number the eigenvalues $(\lambda_{k,j}(a) )_{k \in \bbZ, j \in \bbZ}$ of $\Da + \Za$ such that for all $\eps > 0$ there is an $A \geq 0$ such that for all $a \geq A$
\begin{align*}
\vert \lambda_{k,j}(a)\vert \geq \sinh \left( \arsinh \left(\frac{k}{l_a} - \frac{1}{2}\left[ \frac{n}{2} \right]^{\frac{1}{2}} C_A  - \eps \right) - \eps \right) - \Lambda.
\end{align*}
In particular, as $\lim_{a \rightarrow \infty} \la = 0$ all eigenvalues $\lambda_{k,j}(a)$ with $k \neq 0$ diverge as $a$ tends to infinity.

If in addition $\Vert \Za \Vert_{W^{1, \infty}}\leq \Lambda$ holds for all $a \in \bbN$. Then, the eigenvalues $\lambda_{0,j}(a)$ of $\Da+\Za$ converge to the eigenvalues of the operator
\begin{align*}
D^N + \frac{i}{4} \omega_n^{\bbC} \gamma(\mathcal{F}) + \mathcal{Z}, &\qquad \text{if $n$ is even,} \\
& \\
\begin{pmatrix*}
D^N + \mathcal{Z}^{++}&  \frac{i}{4} \gamma(\mathcal{F})+ \mathcal{Z}^{-+} \\
\frac{i}{4} \gamma(\mathcal{F})+ \mathcal{Z}^{+-} & -D^N + \mathcal{Z}^{--}
\end{pmatrix*},
&\qquad \text{if $n$ is odd.}
\end{align*}

If $N$ is orientable then $D^N$ is the Dirac operator on $\Sigma N$,  $\mathcal{Z}$ is a $W^{1, \infty}$-operator on $\Sigma N$, resp.\ $\Sigma^+N \oplus \Sigma^-N$, $\omega^{\bbC}_n$ is the complex volume element of $\Sigma N$, resp.\ $\Sigma^+N \oplus \Sigma^-N$ and $\mathcal{F}$ is the limit two-form of the sequence $(\cFa)_{a \in \bbN}$, where $f_a^{\ast} \cFa = \la \Fa$.

If $N$ is nonorientable then $D^N$ is the twisted Dirac operator on the twisted \pin \ bundle $\Sigma^P \otimes \mathcal{K}^{\bbC}$, where $\mathcal{K}^{\bbC}$ is the complexified determinant bundle, and $\mathcal{F}$ is the limit two-form of the sequence $(\mathcal{F}_a)_{a \in \bbN} \subset \Omega^2(N, \mathcal{K})$ where $f_a^{\ast} \mathcal{F}_a = -2 A_a$ and the limit objects $\mathcal{Z}$ and $\omega_n^{\bbC}$ act on the twisted \pin \ bundles  $\Sigma^P N \otimes \mathcal{K}^{\bbC}$, resp.\ $(\Sigma^{P+}N \oplus \Sigma^{P-}N) \otimes \mathcal{K}^{\bbC}$
\end{thm}

\begin{proof}
The proof for the lower bound on the eigenvalues $\lambda_{k,j}(a)$ of the operator $\Da + \Za$ is similar to the proof of Theorem \ref{nonprojectableEigenvalues}. The only change lies in the fact that $k$ takes now values in $\bbZ$ instead of $(\bbZ + \frac{1}{2})$.

For the second part of the theorem, we first pass to invariant metrics $\tga$ such that $\bbS^1$ acts on $(\Ma, \tga)$ by isometries and $\lim_{a \rightarrow \infty} \Vert \ga - \tga \Vert_{C^1} = 0$, see Theorem \ref{invariantMetric}. Observe that therefore
\begin{align*}
\lim_{a \rightarrow \infty} \Vert \Da - \tilde{D}_a \Vert = 0
\end{align*}
and thus
\begin{align*}
\lim_{a \rightarrow \infty}\dist (\sigma (\Da), \sigma(\tilde{D}_a ) ) = 0.
\end{align*}
Hence, we can assume without loss of generality that the metrics $\ga$ of the sequence $(\Ma, \ga)_{a \in \bbN}$ are all invariant.

For each $a \in \bbN$, recall the associated operator $\tZa$ which is $\bbS^1$-invariant, i.e.\ $\tZa (V_0(a) ) \subset V_0(a)$. Furthermore, it follows with Proposition \ref{ConvergenceAssOp} that
\begin{align*}
\lim_{a \rightarrow \infty} \Vert \left.\Za\right._{\vert V_0(a)} - \left. \tZa \right._{\vert V_0(a)} \Vert_{L^{\infty}} = 0
\end{align*}
as $\Vert \Za \Vert_{W^ {1, \infty}} \leq \Lambda$ for all $a \in \bbN$. 

If $N$ is orientable, we consider the operator
\begin{align*}
\left. (\Da+ \tZa) \right._{\vert V_0(a)} & = Q_{0,a} \circ D^N_a \circ Q_{0,a}^{-1} + \frac{i^2}{4}\gamma\left( \frac{K}{\la}\right) \gamma(\la \Fa)  + \tZa \\
&=
\begin{cases}
Q_{0,a} \circ ( D^N_a + \frac{i}{4} \omega_{n,a}^{\bbC} \gamma(\cFa) + \cZa ) \circ Q_{0,a}^{-1} , & \text{if $n$ is even} \\
Q_{0,a} \circ \begin{pmatrix*}
D^N_a + \cZa^{++} &  \frac{i}{4} \gamma(\cFa) + \cZa^{-+} \\
\frac{i}{4} \gamma(\cFa) + \cZa^{+-} & - D^N_a + \cZa^{--}
\end{pmatrix*} \circ Q_{0,a}^{-1} , &\text{if $n$ is odd}
\end{cases}
\end{align*}

Studying each part separately we first observe that $D^N_a$ acting on $L^2(\Sigma_a N)$ converge in norm to the operator $D^N$ acting on $L^2(\Sigma N)$ as the quotient metric $\tilde{h}_a$ on $N$ converge in $C^1$. This also concludes the convergence for $D^N_a$ in the odd dimensional case. For the same reason, the complex volume element $\omega_{n,a}^{\bbC}$ converge to the complex volume element on $\Sigma N$. 

By Corollary \ref{ConvergenceATensor} it also follows that there is a subsequence such that the two forms $\cFa$ converge to a continuous two-form $\mathcal{F}$ on $N$.

Finally, by the assumption that the operators $\Za$ are uniformly bounded in $W^{1, \infty}$ it follows that the same bounds hold for the induced operators $\cZa$ as $\bbS^1$ acts by isometries. Hence, there is a further subsequence, such that $\cZa$ converge in $L^{\infty}$ to a limit operator $\mathcal{Z}$ in $W^{1, \infty}$.

Putting this all together it follows that sequence $  \left( (D_a + \Za)_{\vert V_0(a)} \right)_{a \in \bbN}$ induces a sequence of operators on $N$ that converge in norm to the claimed limit operator. As $Q_{0,a}$ is for all $a \in \bbN$ an isometry the claim follows.

If $N$ is nonorientable we have slightly different representation of $\left. (\Da+ \tZa) \right._{\vert V_0(a)}$ since there is no globally well-defined unit vertical vector field. In that case we can write
\begin{align*}
\left. (\Da+ \tZa) \right._{\vert V_0(a)} & =  Q^P_{0,a} \circ D^N_a \circ \left(Q^P_{0,a}\right)^{-1} - \frac{i^2}{2}\gamma(\tilde{A}_a) + \tZa\\
&=
\begin{cases}
Q^P_{0,a} \circ ( D^N_a + \frac{i}{4} \omega_{n,a}^{\bbC} \gamma(\cFa) + \cZa ) \circ \left(Q^P_{0,a}\right)^{-1} , & \text{if $n$ is even} \\
Q_{0,a} \circ \begin{pmatrix*}
D^N_a + \cZa^{++} &  \frac{i}{4} \gamma(\cFa) + \cZa^{-+} \\
\frac{i}{4} \gamma(\cFa) + \cZa^{+-} & - D^N_a + \cZa^{--}
\end{pmatrix*} \circ \left(Q^P_{0,a}\right)^{-1} , &\text{if $n$ is odd}
\end{cases}
\end{align*}
Here $\tilde{A}_a$ is the restriction of $\tilde{A}_a$ to $\mathcal{H} \times \mathcal{H}$, where $\mathcal{H}$ is the horizontal distribution and $\mathcal{F}_a$ is a two forms with values in the determinant bundle $\mathcal{K}$ such that $f_a^{\ast} \mathcal{F}_a = -2 \tilde{A}_a$. Then the claim follows completely analogous to the oriented case with the slight modification that the convergence of the two forms $\mathcal{F}_a$ follows from Corollary \ref{NonorientableConvergenceATensor}.

\end{proof}

\begin{rem}\label{HigherCodimension}
We have concentrated on the case of codimension one collapse as for higher codimension the limit space can have singularities where the sectional curvatures are unbounded and thus also the $A$-tensor, see \cite[Theorem 0.9]{FukayaBoundary} and \cite[Theorem 1.2]{NaberTian}. Then one could ask if the same strategy would work if we assume the limit space to be a Riemannian manifold. However, by \cite[Proposition 1.1]{MoroianuDiss} the Dirac operator only maps projectable spinors to projectable spinors if and only if the structural group is abelian. Therefore we do not have this characterization in the case of collapsing infranil bundles that are covered by a non-abelian nilpotent group. We are confident that the same strategy should work, after a few modifications, in the setting of flat fiber bundles and hope that this can be in the end generalized to the case of smooth limits in any codimension.
\end{rem}

\section{Discussion for the Dirac Operator without a potential}

In this section we consider the results for the behavior of the spectrum of the Dirac operator under collapse of codimension one without an additional potential and discuss the differences to the results in \cite{Ammann} and \cite{LottDirac}. Consider a collapsing sequence $(\Ma, \ga)_{a \in \bbN}$ of spin manifolds in $\mfdspaceCnn$ converging to an $n$-dimensional Riemannian orbifold $(N,h)$. For simplicity, we assume that the metrics $\ga$ are always invariant in the sense of Theorem \ref{invariantMetric}. As a corollary from Theorem \ref{nonprojectableEigenvalues} and Theorem \ref{projectableEigenvalues} we obtain

\begin{cor}\label{ClassicDirac}
Let $(\Ma, \ga)_{a \in \bbN}$ be a collapsing sequence of spin manifolds in $\mfdspaceCnn$ with limit space $N$. Then we can number the eigenvalues $( \lambda_{k,j}(a))_{k,j}$ of the Dirac operator $\Da$, where $j \in \bbZ$ and $k \in \bbZ$ if the spin structure on $\Ma$ is projectable and $k \in \left( \bbZ + \frac{1}{2} \right)$ if the spin structure on $\Ma$ is nonprojectable, such that
\begin{align*}
\lim_{a \rightarrow \infty} \lambda_{k,j}(a) = \begin{cases} \pm \infty \ & \text{if} \ k \neq 0. \\
\mu_j & \text{if} \ k = 0 ,
\end{cases}
\end{align*}
where $\mu_j$ are the eigenvalues of operator
\begin{align*}
D^N + \frac{i}{4}  \omega_n^{\bbC} \gamma(\mathcal{F}), &\qquad \text{if $n$ is even,} \\
& \\
\begin{pmatrix*}
D^N& + \frac{i}{4} \gamma(\mathcal{F}) \\
\frac{i}{4} \gamma(\mathcal{F}) & -D^N 
\end{pmatrix*},
&\qquad \text{if $n$ is odd.}
\end{align*}

\end{cor}

Thus, if $N$ is orientable and $\mathcal{F} = 0$ we recover the result of \cite[Theorem 3.1 and Theorem 4.1]{Ammann} under weaker assumptions as the uniform bound on the $T$-tensor implies that $\limsup_{a \rightarrow \infty} \Vert \grad \la \Vert = 0$, where $\la: N \rightarrow \bbR_+$ is the length of the fibers of the $\bbS^1$-principal bundle $f_a: \Ma \rightarrow N$, whereas Ammann only requires that $\limsup_{a \rightarrow \infty} \Vert \grad \la \Vert \leq 1$ if the spin structures are projectable and $\limsup_{a \rightarrow \infty} \Vert \grad \la \Vert \leq \frac{1}{2}$, if the spin structures are nonprojectable. However, if one wants to consider general collapsing sequences in $\mfdspaceCnn$ one has to deal with the case $\mathcal{F} \neq 0$ which causes the perturbation of the Dirac operator in the limit.

\begin{ex}\label{NonVanishingA}
Let $n$ be an even number and consider a fixed non flat $\bbS^1$-bundle $f: (M^{n+1},g) \rightarrow (N^n,h)$ such that $f$ is a Riemannian submersion with totally geodesic fibers of constant length $2\pi$. Denote by $F = f^{\ast} \mathcal{F}$ the curvature of the bundle. Suppose that $M$ is endowed with a projectable spin structure. Consider for each $k \in \bbN$ the cyclic subgroup $\bbZ_k  < \bbS^1$.

Set $M_k \coloneqq \faktor{M_k}{\bbZ_k}$. By construction, there is a well-defined quotient metric $g_k$ on $M_k$. We have that $\lim_{k \rightarrow \infty} (M_k, g_k) = (N,h)$ is a collapse under bounded curvature. Observe that the length of the fibers scale like $l_k = \frac{2 \pi}{k}$ and the curvature like $F_k = k F$. Recall the isometry $Q: L^2(\Sigma N) \rightarrow V_0$. Then
\begin{align*}
\left.D_k\right._{\vert V_0} &= Q \circ D^N \circ Q^{-1} - \frac{1}{4} \gamma\left( \frac{K_k}{l_k} \right) \gamma(l_k F_k) \\
&= Q \circ \left( D^N - \frac{i}{4} \omega_n^{\bbC} \gamma\left( \frac{2\pi}{k} k F \right)  \right) \circ Q^{-1} \\
&= Q \circ (D^N - \frac{i}{4} \omega_n^{\bbC} \gamma( 2 \pi F) ) \circ Q^{-1}. 
\end{align*} 
Therefore we see that the spectrum of the Dirac operator $D_k$ restricted to $V_0$ equals the spectrum of $( D^N - \frac{i}{4} \omega_n^{\bbC} \gamma(2 \pi F) )$ for all $k \in \bbN$.
\end{ex}

In \cite{LottDirac} the behavior of Dirac eigenvalues under collapse with bounded curvature were discussed in great generality. Lott considered collapse of any codimension to smooth and singular limit spaces \cite[Theorem 2 - 4]{LottDirac} and the behavior of Dirac eigenvalues on any $\mathrm{G}$-Clifford bundles, where $\mathrm{G}= \SO(n)$ or $\mathrm{G} = \Spin(n)$. Therefore, his results also includes the Dirac operator on differential forms. 

In this article we restrict ourselves to the setting of the $\Spin$-Clifford bundle induced by the canonical spin representations and to collapse of codimension one. Due to this restriction we obtain the following accentuation of Lotts results:  Let $(\Ma, \ga)_{a \in \bbN}$ be a sequence in $\mfdspaceCnn$ collapsing to an $n$-dimensional Riemannian orbifold $N$. We suppose further that, for $a$ large enough, there is a Riemannian submersion $f_a: M \rightarrow N$. In the case of nonprojectable spin structures the results of Corollary \ref{ClassicDirac} coincides with those of \cite[Theorem 4]{LottDirac}. In the case of projectable spin structures it is shown in \cite[Theorem 2, Theorem 3]{LottDirac} that the spectra of the Dirac operators $\Da$ acting on the spinors of $\Ma$ converges to the spectrum of a first order differential operator $D$. Here $D^2$ is the sum of the Laplacian on $L^2(N, \chi \dvol)$ for some function $\chi$ and a zero-order term depending on the limit of the curvature operators on $\Ma$. In our restricted setting, we have shown in Corollary \ref{ClassicDirac} that $\chi \equiv 1$ and that $D$ is in fact the Dirac operator on the limit space together with a zero-order potential depending on the sequence of the integrability tensors of the Riemannian submersions $f_a: \Ma \rightarrow N$. In the following example we show that in the general case, e.g.\ for the Dirac operator on differential form, the choice of $\chi$ is nontrivial and that the limit of the Dirac spectra depends on the second fundamental form of the fibers, in contrast to the spin case.

\begin{ex}
Consider the torus $T^2 = \lbrace ( e^{is}, e^{it}) : s,t \in \bbR \rbrace$ with the Riemannian metric
\begin{align*}
g_{\eps} \coloneqq \mathrm{d}s^2 \oplus 
\eps^2 c(s)^2 \mathrm{d}t^2,
\end{align*}
for some positive function $c: \bbS^1 \rightarrow \bbR_+$. Then $\lim_{\eps \rightarrow 0} (T^2, g_{\eps} ) = ( \bbS^1, \mathrm{d}s )$. Note that the integrability tensor $A_{\eps} = 0$ for all $\eps$ but the $T$-tensor is characterized by $\frac{c^{\prime}(s)}{c(s)}$, see Lemma \ref{Tensors}. 

Endow $(T^2, g_{\eps})$ with the spin structure induced by the pullback of a chosen spin structure on $\bbS^1$. This defines a projectable spin structure on $(T^2, g_{\eps})$. By Theorem \ref{projectableEigenvalues} the spectrum of the Dirac operator $D_{\eps}$ on $(T^2, g_{\eps})$ restricted to the $\bbS^1$-invariant spinors, i.e.\ $\mathcal{L}_{\partial_t} \varphi = 0$, converge to the spectrum of the Dirac operator $D^{\bbS^1}$ on $\bbS^1$.

Now we take a look on the Dirac operator on forms. In  that case the space of ``projectable forms" is given by
\begin{align*}
V_0 \coloneqq \lbrace f \in C^{\infty}(T^2) : \frac{\partial}{\partial t} f = 0 \rbrace \cup \lbrace \alpha \, \mathrm{d}s \in \Omega^1(T^2) : \frac{\partial}{\partial t} \alpha = 0 \rbrace
\end{align*}
we take some $(f + \alpha \mathrm{d}s) \in V_0$ and calculate that the Dirac operator $\D_{\eps} = \mathrm{d} + \delta$ acts on it as
\begin{align*}
\D_{\eps}(f(s) + \alpha(s) \mathrm{d}s) = \frac{\partial}{\partial s} f \mathrm{d}s - c(s)^{-1} \frac{\partial}{\partial s}(c(s) \alpha(s) ). 
\end{align*}
Observe that as $\eps$ goes to zero $\D_{\eps}$ converges to a first order differential operator $\D_0$ acting on $\Omega^{\ast}(\bbS^1)$. 

Suppose that $c(s) = \exp(g(s))$ is a smooth function on $\bbS^1$ whose exponent is given by $g(s) = \sum_{k=1}^{\infty}  a_k \cos(kt)$, with $a_k \geq -2$ for all $k$. Then one can check that 
\begin{align*}
\sigma(\D_0) = \left\lbrace \pm k \sqrt{1 + \frac{a_k}{2}} : k \in \bbN \right\rbrace \neq \sigma(\D^{\bbS^1} ).
\end{align*}
\end{ex}

Note that Corollary \ref{ClassicDirac} also holds if the limit space is a Riemannian orbifold. Comparing with \cite[Theorem 3]{LottDirac}, we obtain, in this case, a convergence of $\sigma(\Da)$ to $\sigma(D)$ instead of $\sigma(\vert \Da \vert)$ to $\sigma(\vert D \vert)$. 

The main difference of the strategy between \cite{LottDirac} and this article is that in \cite{LottDirac} the results were proven by using the Schr\"odinger-Lichnerowicz formula $D^2 = \nabla^{\ast} \nabla + V$, where $V$ is some potential depending on the $G$-Clifford bundle and the curvature of the manifold, and the results in \cite{LottLaplaceSmooth}, \cite{LottLaplaceSingular} concerning the eigenvalues of Laplace operators on collapsing manifolds. In this article, we did not use the Schr\"odinger-Lichnerowicz formula, but the isometry $Q_k$, see Section 3, to relate the Dirac operator on the manifold to the Dirac operator on the limit space.

We do not know yet if our results can be generalized to collapse of higher codimensions as there are various things to consider additionally, see Remark \ref{HigherCodimension}. However, we believe that it should be possible to modify this strategy to orbifold limit spaces to obtain similar results to Corollary \ref{ClassicDirac} for collapse of higher codimension.

\bibliography{literature.bib}
\bibliographystyle{amsalpha}

\end{document}